\documentclass[review]{elsarticle}

\usepackage{lineno,hyperref}
\modulolinenumbers[5]
\usepackage{amssymb}
\usepackage{amsmath}
\usepackage{amsfonts}
\usepackage{latexsym}

\usepackage{graphics}
\usepackage{float}
\usepackage{epstopdf}

\newtheorem{theorem}{Theorem}

\newtheorem{corollary}[theorem]{Corollary}

\newenvironment{proof}[1][Proof]{\noindent\textbf{#1.} }{\ \rule{0.5em}{0.5em}}

\journal{Journal of Mathematical Analysis and Applications}

\begin{document}

\begin{frontmatter}

\title{Novel series involving families of zeta functions and special functions and polynomials by applying N\"{o}rlund sum and derivative formula}

\author[firstauthor]{Elif Sükrüoglu \corref{mycorrespondingauthor}}
\address[firstauthor]{Department of Mathematics, Faculty of Science Akdeniz University of
	Akdeniz TR-07058 Antalya, Turkey}
\cortext[mycorrespondingauthor]{Corresponding author}
\ead{elifsukruoglu@akdeniz.edu.tr}

\author[firstauthor]{Yilmaz Simsek}

\begin{abstract}
In recent years, the N\"{o}rlund sum and its applications have attracted considerable attention, especially in mathematics and other applied sciences. Because this sum plays an important role in mathematical modeling involving differential equations and difference equations, as well as in the theory of infinite series and integral transforms. Therefore, one of the most important goals of this article is to present a new approach to generating function  and to derive new formulas including the sum of (inverse) Laplace transforms, the N\"{o}rlund sum, the alternative Hurwitz zeta function, and the inverse Catalan sum formula with related operators. Another aim is to establish the relationship between the inverse Catalan sum formula and the derivative operator, and to give new formulas and relations, including the  zeta function, the Bernoulli numbers, and the Laplace transform. Furthermore, applications are given including the relationship between these operators and the Dirichlet $L$-function. Some new formulas for certain classes of infinite series, including the Dirichlet $L$-function, the zeta function,  Euler numbers of the second kind, and the Apostol-Bernoulli polynomials of higher-order.
\end{abstract}

\begin{keyword}
 N\"{o}rlund sum, Laplace transform, Dirichlet $L$-function, finite sums, special functions and polynomials, Hurwitz zeta functions, derivative operator
\MSC[2020] 05A15\sep  11B68\sep 11M35\sep 35J05\sep 39A70\sep 47E07\sep 47B37 \sep 60E05.

\end{keyword}

\end{frontmatter}

\linenumbers
\section{Introduction}
\label{Sec:1}
Special polynomials and numbers involving the Bernoulli polynomials, the
Euler polynomials, and Apostol polynomials, and others have long played
significant roles in various areas. These polynomials are not only of
theoretical interest, but also have useful applications in\textbf{\ }many
different areas such as generating functions, number theory, combinatorics,
probability, statistics, and even quantum physics. Their applications can
also appear in the development of novel generating functions, sum formulas,
finite and infinite sums, recurrence relations, and computational methods.
These also reveals their algebraic and analytical properties. On the other
hand, these polynomials have many applications in the theory of the N\"{o}%
rlund sum, which is a generalization of the usual summation method in
discrete analysis. The N\"{o}rlund sum provides a very useful bridge between
the family of special numbers and polynomials, the sum of (inverse) Laplace
transforms, and infinite series. Moreover, developed as a method for
sequence transformation and summability, the N\"{o}rlund sum has gained an
important place in difference equations, approximation theory, and the
construction generating functins for special functions and polynomials. In
particular, the use of the N\"{o}rlund sum allows to obtain identities of
polynomials that cannot be easily derived by usuall methods. Furthemore, the
N\"{o}rlund sum can be used in the theory of divergent or convergent series.

The motivation of this paper is to give many new formulas and relations
involving finite and infinite sums, the Dirichlet $L$-function attached to
Dirichlet character with positive integer conductor, the Hurwitz zeta
function, the alternating Hurwitz zeta function, Bernoulli polynomials, and
Apostol-Bernoulli polynomials of higher order, etc. with the aid of the N%
\"{o}rlund sum with its properties and the homogenous N\"{o}rlund sum with
its Representation Theorem involving the (inverse) Laplace transform.

The definitions and formulas needed to achieve these motivation are briefly
given as follows:

The N\"{o}rlund sum operator is defined by means of the following infinite
series representation:%
\begin{equation*}
	Q(x;w)=-w\sum_{v=0}^{\infty }\varphi (x+vw),
\end{equation*}%
which satisfies the \ following N\"{o}rlund principal solution:%
\begin{equation}
	\overset{x}{\underset{j}{\boldsymbol{S}}}\varphi (z)\underset{w}{\Delta }%
	z=\int\limits_{j}^{\infty }\varphi (t)dt-w\sum_{v=0}^{\infty }\varphi (x+vw),
	\label{3.10}
\end{equation}%
where $\overset{x}{\underset{j}{\boldsymbol{S}}}\varphi (z)\underset{w}{%
	\Delta }z$\ is summing $\varphi (z)$\ from $j$\ to $x$, $w$ is a span of the
sum.

The fundamental problem we now address concerns finding the solution of the equation
\begin{equation}
	\underset{w}{\Delta }\left\{ Q(x;w)\right\} =\frac{E^{w}-I}{w}Q(x;w)
	\label{3.1e}
\end{equation}%
for given primitives $ Q(x;w)$ corresponding to a known function $\varphi (z)$, where $E^{w}$
denotes the shift operator.
\cite{Legermaqn,MilneThomson,SimsekMMASOperator, SukruogluRevista, SukruogluJMAA}.

By the aid of equation (\ref{3.10}), we have%
\begin{equation}
	\overset{x}{\underset{j}{\boldsymbol{S}}}\varphi (z)\underset{w}{\Delta }%
	z=\int\limits_{j}^{x}\varphi (t)dt+\overset{x}{\underset{b}{\boldsymbol{S}}}%
	\varphi (z)\underset{w}{\Delta }z  \label{3.12}
\end{equation}%
\cite{Legermaqn}. Since%
\begin{equation*}
	\lim_{w\rightarrow 0}\frac{E^{w}-I}{w}\varphi (x)=\varphi ^{^{\prime }}(x),
\end{equation*}%
where%
\begin{equation*}
	\varphi ^{^{\prime }}(x)=\frac{d}{dx}\left\{ \varphi (x)\right\}. 
\end{equation*}%
Then we have
\begin{equation*}
	\lim_{w\rightarrow 0}\overset{x}{\underset{j}{\boldsymbol{S}}}\varphi (z)%
	\underset{w}{\Delta }z=\int\limits_{j}^{\infty }\varphi (t)dt
\end{equation*}%
\cite{Legermaqn}. 

Consequently, by the same method of Jagerman \cite%
{Legermaqn}, we have the following homogeneous form of the N\"{o}%
rlund sum, which is defined by%
\begin{equation}
	\overset{x}{\underset{j}{\boldsymbol{S}}}\varphi (z)\underset{w}{\Delta }%
	z=\int\limits_{j}^{b}\varphi (t)dt+H(x;w),  \label{3.79}
\end{equation}%
where 
\begin{equation*}
	H(x;w)=\overset{x}{\underset{x}{\boldsymbol{S}}}\varphi (z)\underset{w}{%
		\Delta }z
\end{equation*}%
\cite{Legermaqn}.

The Representation Theorem with (inverse) Laplace transforms is given by%
\begin{equation}
	\overset{x}{\underset{j}{\boldsymbol{S}}}\mathcal{L}\left\{ f(z)\right\} 
	\underset{w}{\Delta }z=\int\limits_{0}^{\infty }f(t)\left( \frac{e^{-jt}}{t}-%
	\frac{we^{-tx}}{1-e^{-wt}}\right) dt,  \label{Rt}
\end{equation}%
where $\mathcal{L}\left\{ f(z)\right\} $ denotes Laplace transform of $f(z)$%
, defined by%
\begin{equation*}
	\mathcal{L}\left\{ f(z)\right\} =\int\limits_{0}^{\infty }e^{-zt}f(t)dt,
\end{equation*}%
when $z>0$ \cite{Bateman,Legermaqn,MilneThomson}.

The Bernoulli polynomials are defined by 
\begin{equation*}
	\frac{te^{tx}}{e^{t}-1}=\sum_{n=0}^{\infty }\frac{B_{n}(x)}{n!}t^{n},
\end{equation*}%
with $B_{0}(x)=1$ \cite{SIMSEKMONTES,SrivastavaChoi}.

The Lerch zeta function is defined by%
\begin{equation}
	\Phi (z,s,a)=\sum_{n=0}^{\infty }\frac{z^{n}}{(n+a)^{s}},  \label{LZ}
\end{equation}%
where $a\in 
\mathbb{C}
\backslash 
\mathbb{Z}
_{0}^{-}$; $s\in 
\mathbb{C}
$ when $\left\vert z\right\vert <1$; $Re(s)>1$ when $\left\vert z\right\vert
=1 $ \cite{SrivastavaChoi}. For $z=-1$ and    $Re(s)>1$ the alternating
Hurwitz zeta functions defined as follows:%
\begin{equation}
	\zeta _{E}\left( s,a\right) :=\Phi (-1,s,a)=\sum_{n=0}^{\infty }\frac{\left(
		-1\right) ^{n}}{(n+a)^{s}},  \label{Alterne}
\end{equation}%
with $b>1$ \cite{SrivastavaChoi}.

Let $\chi $ be a Dirichlet character with conductor $k$. Let $s=x+iy$ with $%
x>1$. The Dirichlet $L$-function is defined by%
\begin{equation}
	L\left( s,\chi \right) =\sum_{n=0}^{\infty }\frac{\chi \left( n\right) }{%
		n^{s}},  \label{El-1}
\end{equation}%
\cite{Apostol,SrivastavaChoi}.

The Hurwitz zeta function is defined by%
\begin{equation*}
	\zeta (s,a)=\sum_{n=0}^{\infty }\frac{1}{(n+a)^{s}},
\end{equation*}%
where $s=\alpha +i\beta $ and $\alpha >1$ \cite{Apostol,KimSimsek,SrivastavaChoi}.

Here $\alpha $ is a fixed real number. When $\alpha =1$ this reduuces to the
Riemann zeta function%
\begin{equation*}
	\zeta (s)=\zeta (s,1).
\end{equation*}%
The function $L\left( s,\chi \right) $ can be expressed in terms of Hurwitz
zeta function%
\begin{equation*}
	L\left( s,\chi \right) =k^{-s}\sum_{v=0}^{k-1}\chi \left( v\right) \zeta
	\left( s,\frac{v}{k}\right) ,
\end{equation*}%
\cite{Apostol,Ireland}.

Substituting $\chi (-1)=-1$ with conductor $k\geq 3$ and $s=1$\ into (\ref%
{El-1}), we have%
\begin{equation*}
	L\left( 1,\chi \right) =\frac{\pi }{2k}\sum_{v=0}^{k-1}\chi \left( v\right)
	\cot \left( \frac{\pi v}{k}\right) .
\end{equation*}%
For $k=4$, the above sum reduces to the following known result:%
\begin{eqnarray}
	L\left( 1,\chi \right)  &=&\frac{\pi }{8}\left( \chi \left( 1\right) \cot
	\left( \frac{\pi }{4}\right) +\chi \left( 2\right) \cot \left( \frac{\pi }{2}%
	\right) +\chi \left( 3\right) \cot \left( \frac{3\pi }{4}\right) \right) 
	\label{El-5} \\
	&=&\frac{\pi }{8}(1-(-1))  \notag \\
	&=&\frac{\pi }{4}.  \notag
\end{eqnarray}%
The Dirichlet beta function is defined by%
\begin{equation*}
	\beta _{n}\left( s\right) =\sum_{k=0}^{\infty }\frac{(-1)^{k}}{(2k+1)^{s}},
\end{equation*}%
where $Re(s)>0$ \cite{Apostol,Ireland,SrivastavaChoi}.
\begin{equation}
	\beta _{n}\left( 2n+1\right) \label{ess1} =(-1)^{n}\frac{E^{*}_{2n}\pi ^{2n+1}}{2^{2n+2}(2n)!},
\end{equation}
where $E^{*}_{2n}$ denotes the Euler numbers of the second kind, which were given in \cite{Comtet,DC}.

Other important results discovered in this paper are also briefly summarized
as follows:

In Section 2, we give novel formulas and results for interpolation functions
of Bernoulli polynomials derived from the N\"{o}rlund sum. Furthermore, we
give the relationships between these functions and the Alternating Hurwitz
zeta function, the reciprocal Catalan sum formula, and the derivative
operator. We also obtain new results by establishing the relationship
between these functions.

In Section 3, we derive new formulas by exploiting the relationship between
the alternating zeta function and the Hurwitz zeta function derived from the
N\"{o}rlund sum. We establish the relationship between the Bernoulli
polynomials of these operators. We also obtain novel results by exploiting
the Dirichlet $L$-function and Dirichlet characters. We use the Representation Theorem for the N\"{o}rlund sum to
give applications of the (inverse) Laplace for some special functions. We
also obtain new results for the homogeneous form of the N\"{o}rlund sum.

In Section 4, we give a conclusion section.
\section{Formulas for interpolation functions of the Bernoulli polynomials
	derive from N\"{o}rlund sum}

In this section, we give some novel formulas derived from the N\"{o}rlund
sum. We also derive the relationships between these formulas and Bernoulli
polynomials, the derivative operator, and the reciprocal Catalan sum
formula. Moreover, by applying the N\"{o}rlund sum to the following function%
\begin{equation*}
	\varphi (z)=\frac{1}{z^{2}+1},
\end{equation*}%
where $z\neq i$, we give another representation for the Bernoulli
polynomials and the alternating Hurwitz zeta function using the N\"{o}rlund
sum and the sum of Laplace transform.

\begin{theorem}
	The following relationship holds true:%
	\begin{equation}
		\overset{x}{\underset{1}{\boldsymbol{S}}}\frac{1}{z^{2}+1}\underset{w}{%
			\Delta }z=\frac{\pi }{4}-\sum_{n=0}^{\infty }\frac{(-1)^{n}}{w^{2n+1}}\zeta
		_{E}\left( 2n+2,\frac{x}{w}\right) .  \label{HZ}
	\end{equation}
\end{theorem}

\begin{proof}
	Substituting%
	\begin{equation*}
		\varphi (z)=\frac{1}{z^{2}+1}
	\end{equation*}%
	into the Representation Theorem (or the sum of Laplace transform), given in 
	\cite{Legermaqn}, we get%
	\begin{equation}
		\overset{x}{\underset{1}{\boldsymbol{S}}}\frac{1}{z^{2}+1}\underset{w}{%
			\Delta }z=\int\limits_{0}^{\infty }\left( \frac{e^{-t}}{t}-\frac{we^{-xt}}{%
			1-e^{-wt}}\right) \sin tdt.  \label{Rt-1}
	\end{equation}%
	Combining the above equation with%
	\begin{equation*}
		\sin t=\sum_{n=0}^{\infty }(-1)^{n}\frac{t^{2n+1}}{(2n+1)!},
	\end{equation*}%
	we get%
	\begin{eqnarray*}
		\overset{x}{\underset{1}{\boldsymbol{S}}}\frac{1}{z^{2}+1}\underset{w}{%
			\Delta }z &=&\int\limits_{0}^{\infty }\frac{e^{-t}}{t}\sin tdt \\
		&&-w\sum_{n=0}^{\infty }\sum_{m=0}^{\infty }\frac{(-1)^{n}}{(2n+1)!}%
		\int\limits_{0}^{\infty }t^{2n+1}e^{-t(mw+x)}dt.
	\end{eqnarray*}%
	After some calculation, and using the following well-known the Laplace
	formula for the function $\frac{\sin t}{t}$:%
	\begin{equation*}
		\int\limits_{0}^{\infty }\frac{e^{-t}}{t}\sin tdt=\frac{\pi }{4},
	\end{equation*}%
	yields%
	\begin{equation*}
		\overset{x}{\underset{1}{\boldsymbol{S}}}\frac{1}{z^{2}+1}\underset{w}{%
			\Delta }z=\frac{\pi }{4}-\sum_{n=0}^{\infty }\frac{(-1)^{n}}{w^{2n+1}}%
		\sum_{m=0}^{\infty }\frac{1}{(m+\frac{x}{w})^{2n+2}}.
	\end{equation*}%
	Combining the above equation with (\ref{Alterne}), we also get%
	\begin{equation*}
		\overset{x}{\underset{1}{\boldsymbol{S}}}\frac{1}{z^{2}+1}\underset{w}{%
			\Delta }z=\frac{\pi }{4}-\sum_{n=0}^{\infty }\frac{(-1)^{n}}{w^{2n+1}}\zeta
		_{E}\left( 2n+2,\frac{x}{w}\right) .
	\end{equation*}%
	The proof is completed.
\end{proof}

Putting the following the reciprocal Catalan sum formula, given by \cite{Edgar,GunSimsek}:%
\begin{equation*}
	\sum_{k=0}^{\infty }\frac{2^{k}}{C_{k}}=5+6\arctan (1)
\end{equation*}%
in (\ref{HZ}), we arrive at the following corollary:

\begin{corollary}
	The following relationship holds true:%
	\begin{equation}
		\overset{x}{\underset{1}{\boldsymbol{S}}}\frac{1}{z^{2}+1}\underset{w}{%
			\Delta }z=\frac{1}{6}\sum_{k=0}^{\infty }\frac{2^{k}}{C_{k}}%
		-\sum_{n=0}^{\infty }\frac{(-1)^{n}}{w^{2n+1}}\zeta _{E}\left( 2n+2,\frac{x}{%
			w}\right) -\frac{5}{6}.  \label{El-2}
	\end{equation}
\end{corollary}

\begin{theorem}
	Let $x\in \mathbb{R}$. Then we have%
	\begin{equation}
		\overset{x}{\underset{1}{\boldsymbol{S}}}\frac{1}{z^{2}+1}\underset{w}{%
			\Delta }z =\frac{\pi }{4}  \label{BB} \\
		-\sum_{v=1}^{\infty }w^{v}\frac{B_{v}(x;w)}{v!}\frac{d^{v-1}}{dx^{v-1}}%
		\left( \frac{1}{1+x^{2}}\right) .  
	\end{equation}
\end{theorem}

\begin{proof}
	Putting the following Laplace formulas%
	\begin{equation*}
		L\left\{ \frac{e^{-t}}{t}\sin t\right\} =\cot ^{-1}(1)=\frac{\pi }{4}
	\end{equation*}%
	and%
	\begin{equation*}
		L\left\{ \sin t\right\} =\frac{1}{x^{2}+1}
	\end{equation*}
	in (\ref{Rt-1}), we get%
	\begin{eqnarray*}
		\overset{x}{\underset{1}{\boldsymbol{S}}}\frac{1}{z^{2}+1}\underset{w}{%
			\Delta }z &=&\frac{\pi }{4}-\frac{wB_{1}(x;w)}{1!}\int\limits_{0}^{\infty
		}\sin te^{-xt}dt \\
		&&+\sum_{v=2}^{\infty }(-w)^{v}\frac{B_{v}(x;w)}{v!}\int\limits_{0}^{\infty
		}t^{v-1}\sin te^{-xt}dt,
	\end{eqnarray*}%
	which is reducesed to%
	\begin{eqnarray}
		\overset{x}{\underset{1}{\boldsymbol{S}}}\frac{1}{z^{2}+1}\underset{w}{%
			\Delta }z &=&\frac{\pi }{4}-\frac{wB_{1}(x;w)}{1+x^{2}}  \label{El-3} \\
		&&-\sum_{v=2}^{\infty }w^{v}\frac{B_{v}(x;w)}{v!}\frac{d^{v-1}}{dx^{v-1}}%
		\left( \frac{1}{1+x^{2}}\right),  \notag
	\end{eqnarray}%
	since%
	\begin{equation*}
		\int\limits_{0}^{\infty }t^{v-1}\sin te^{-xt}dt=\frac{d^{v-1}}{dx^{v-1}}%
		\left( \frac{1}{1+x^{2}}\right) .
	\end{equation*}
	
	Thus, after some calculations, proof of theorem is completed.
	
\end{proof}

\begin{theorem}
	Let $x\in \mathbb{R}$. Then we have%
	\begin{eqnarray*}
		\overset{x}{\underset{1}{\boldsymbol{S}}}\frac{1}{z^{2}+1}\underset{w}{%
			\Delta }z &=&\frac{\pi }{4}-\frac{2x-w}{2\left( 1+x^{2}\right) }%
		-\sum_{v=2}^{\infty }w^{v}\frac{B_{v}(x;w)}{v!} \\
		&&\times \sum_{j=0}^{v}\sum_{n=0}^{\infty }\sum_{j=0}^{n}\sum_{c=0}^{\left[ 
			\frac{n-1-j}{2}\right] }\frac{(-1)^{j+n+c+1}}{\left( w+j\right) ^{n+1}}%
		\binom{n}{j}\binom{n-j}{2c+1}\binom{v}{j} \\
		&&\times \lambda ^{j}x^{n-j-2c-1}\mathcal{B}_{j}^{(v)}\left( \lambda \right)
		.
	\end{eqnarray*}
\end{theorem}

\begin{proof}
	We begin with the following formulas, which are given in \cite{Kilar},%
	\begin{equation*}
		\frac{d^{v-1}}{dx^{v-1}}\left( \frac{y}{y^{2}+x^{2}}\right) =\sum_{j=0}^{v}%
		\binom{v}{j}\lambda ^{j}\sum_{n=0}^{\infty }(-1)^{j+n+1}\frac{1}{\left(
			w+j\right) ^{n+1}}\mathcal{B}_{n}^{(S,v)}\left( x,y;\lambda \right),
	\end{equation*}%
	where%
	\begin{equation*}
		\mathcal{B}_{n}^{(S,v)}\left( x,y;\lambda \right) =\sum_{j=0}^{n}\binom{n}{j}%
		\mathcal{B}_{j}^{(v)}\left( \lambda \right) S_{n-j}(x,y)
	\end{equation*}%
	and%
	\begin{equation*}
		S_{n-j}(x,y)=\sum_{c=0}^{\left[ \frac{n-1-j}{2}\right] }\left( -1\right) ^{c}%
		\binom{n-j}{2c+1}x^{n-j-2c-1}y^{2c+1},
	\end{equation*}%
	for $y=1$, we get
	\begin{eqnarray}
		&&\frac{d^{v-1}}{dx^{v-1}}\left( \frac{1}{1+x^{2}}\right)  \label{e11} \\
		&=&\sum_{j=0}^{v}\binom{v}{j}\lambda   \notag ^{j}\sum_{n=0}^{\infty
		}\sum_{j=0}^{n}\sum_{c=0}^{\left[ \frac{n-1-j}{2}\right] }\frac{%
			(-1)^{j+n+c+1}}{\left( w+j\right) ^{n+1}}\binom{n}{j}\mathcal{B}%
		_{j}^{(v)}\left( \lambda \right) \binom{n-j}{2c+1}x^{n-j-2c-1},
	\end{eqnarray}%
	where $\left\vert \lambda \right\vert <1$. 
	We can rewrite the right hand side of equation (\ref{BB}) as follows
	\begin{equation*}
		\overset{x}{\underset{1}{\boldsymbol{S}}}\frac{1}{z^{2}+1}\underset{w}{%
			\Delta }z =\frac{\pi }{4} \\
		-wB_{1}(x;w)-\sum_{v=2}^{\infty }w^{v}\frac{B_{v}(x;w)}{v!}\frac{d^{v-1}}{dx^{v-1}}%
		\left( \frac{1}{1+x^{2}}\right) .  
	\end{equation*}
	
	Combining the above formula with (\ref{e11}), and using $B_{1}(x;w)=\frac{x}{w}-\frac{1}{2}$, after some
	elementary calculations, we complete proof of theorem.
\end{proof}

By combining (\ref{HZ}) with (\ref{BB}), after some elementary calculations, we arrive at the following theorem:

\begin{theorem}
	Let $x\in \mathbb{R}$. Then we have%
	\begin{eqnarray*}
		&&\sum_{n=0}^{\infty }\frac{(-1)^{n}}{w^{2n+1}}\zeta _{E}\left( 2n+2,\frac{x%
		}{w}\right)  \\
		&=&\frac{2x-w}{2\left( 1+x^{2}\right) }+\sum_{v=2}^{\infty }w^{v}\frac{%
			B_{v}(x;w)}{v!}\sum_{j=0}^{v}\sum_{n=0}^{\infty }\sum_{j=0}^{n}\sum_{c=0}^{%
			\left[ \frac{n-1-j}{2}\right] }\frac{(-1)^{j+n+c+1}}{\left( w+j\right)
			^{n+1}} \\
		&&\times \binom{n}{j}\binom{n-j}{2c+1}\binom{v}{j}\lambda ^{j}x^{n-j-2c-1}%
		\mathcal{B}_{j}^{(v)}\left( \lambda \right) .
	\end{eqnarray*}
\end{theorem}

\section{Infinite series involving alternating zeta function and the Hurwitz
	zeta function}

In this section, we give some new formulas for the N\"{o}rlund sum using the
relationship between the alternating zeta function and the Hurwitz zeta
function. We also derive new results for the Dirichlet $L$-function attached to the Dirichlet characters and the N\"{o}rlund sum using certain special trigonometric
functions. Finally, we give some novel formulas for infinite series
involving alternating zeta function and the Hurwitz zeta function with the
aid of the Dirichlet $L$-function. In order to give these formulas, we use
the following formula:%
\begin{equation}
	\zeta _{E}\left( 2n+2,\frac{x}{w}\right) =\left( 1-2^{-1-2n}\right) \zeta
	\left( 2n+2,\frac{x}{w}\right)  \label{El-4}
\end{equation}%
\cite{Abramowitz,Gradshteyn}.

Combining (\ref{HZ}) with (\ref{El-4}), we get%
\begin{equation*}
	\overset{x}{\underset{1}{\boldsymbol{S}}}\frac{1}{z^{2}+1}\underset{w}{%
		\Delta }z=\frac{\pi }{4}-\sum_{n=0}^{\infty }\frac{(-1)^{n}}{w^{2n+1}}\left(
	1-2^{-1-2n}\right) \zeta \left( 2n+2,\frac{x}{w}\right) .
\end{equation*}

Combining the above equation with the following formula,  given by
\cite{Abramowitz} and \cite{Gradshteyn},%
\begin{equation*}
	\zeta \left( 2n+2,\frac{x}{w}\right) =\frac{(-1)^{n+1}(2\pi )^{2n}}{2(2n)!}%
	B_{2n+2}\left( \frac{x}{w}\right) ,
\end{equation*}

we get the following theorem:

\begin{theorem}
	The following relationship holds true:
	\begin{equation}
		\overset{x}{\underset{1}{\boldsymbol{S}}}\frac{1}{z^{2}+1}\underset{w}{%
			\Delta }z+\sum_{n=0}^{\infty }\frac{(-1)^{2n+1}}{w^{2n+1}}\left( 1-\frac{1}{%
			2^{1+2n}}\right) \frac{(2\pi )^{2n}}{2(2n)!}B_{2n+2}\left( \frac{x}{w}%
		\right) =\frac{\pi }{4}. \label{es-1}
	\end{equation}
\end{theorem}
We now give some special values of the equation (\ref{es-1}) with the aid of the following relations and formulas:
\begin{equation*}
	\frac{\pi}{4} = \sum_{n=0}^{\infty} \frac{(-1)^n}{2n+1} = \arctan(1),
\end{equation*}
\cite{Abramowitz}. It is easy to see that the above relations is related to the function $L\left( 1,\chi \right)$ with the help of the equation (\ref{El-5}) for $k=4,$ which implies
\begin{equation*}
	\chi(n) =
	\begin{cases}
		0, & n\ \text{even}, \\[4pt]
		(-1)^{\frac{n-1}{2}}, & n\ \text{odd}.
	\end{cases}
\end{equation*}
By using the above values of $\chi(n)$, we get $$L\left( 1,\chi \right)=\beta(1).$$ That is, for $\chi_{4}(n):=\chi(n)$ yields $$ \beta(1) = \sum_{n=1}^{\infty} \frac{\chi_{4}(n)}{n}.$$ 

The above relations can also be given by the following Table \ref{t1}:

\begin{table}[H]
	\centering
	\caption{The first terms for the function 	$\beta(1)$ attached to the Dirichlet character $\chi_{4}$.}
	\medskip
	\renewcommand{\arraystretch}{1.2}
	\begin{tabular}{|c|c|c|c|}
		\hline
		$n$ & $\chi_4(n)$ & $\beta(1) =\sum_{n=1}^{\infty}  \frac{\chi_{4}(n)}{n}$ \\ \hline
		1 & $1$ & $1$ \\ 
		2 & $0$ & $0$ \\ 
		3 & $-1$ & $-\dfrac{1}{3}$ \\ 
		4 & $0$ & $0$ \\ 
		5 & $1$ & $\dfrac{1}{5}$ \\ 
		6 & $0$ & $0$ \\ 
		7 & $-1$ & $-\dfrac{1}{7}$ \\ 
		8 & $0$ & $0$ \\ 
		9 & $1$ & $\dfrac{1}{9}$ \\ 
		\hline
	\end{tabular}
	\label{t1}
\end{table}
and so on.

Setting some values of the Euler numbers of the second kind in equation (\ref{ess1}), we obtain  other known values of the function $\beta(n)$, including $\beta(1)=\frac{\pi}{4}$, which are given by Table \ref{t2}:

\begin{table}[h!]
	\caption{The values of the Dirichlet beta function at $\beta(2n + 1)$
	}
	\centering
	\begin{tabular}{|c|c|c|c|c|}
		\hline
		$n$ & $s=2n+1$ & $E^{*}_n$ & $\beta(n)$   \\ \hline
		0 & 1 & $E^{*}_0=1$  & $\displaystyle \beta(1)=\frac{\pi}{4}$ 
		
		\\ [3mm]
		1 & 3 & $E^{*}_2=-1$  & $\displaystyle \beta(3)=\frac{\pi^3}{32}$ 
		\\ [3mm]
		2 & 5 & $E^{*}_4=5$  & $\displaystyle \beta(5)=\frac{5\pi^5}{1536}$ 
		\\ [3mm]
		3 & 7 & $E^{*}_6=-61$  & $\displaystyle \beta(7)=\frac{61\pi^7}{184320}$ 
		\\ [3mm]
		4 & 9 & $E^{*}_8=1385$  & $\displaystyle \beta(9)=\frac{1385\pi^9}{41287680}$ \\  
		\hline
	\end{tabular}
	\label{t2}
\end{table}

Using (\ref{El-5}), it is easy to give a relation between the Dirichlet beta
function and the function $L\left( 1,\chi \right) $:%
\begin{equation}
	L\left( 1,\chi \right) =\frac{\pi }{4}=\beta (1), \label{Ed0}
\end{equation}%
which gives us%
\begin{equation}
	\zeta _{E}\left( 1,\frac{1}{2}\right) =2L\left( 1,\chi \right) \label{Ed1}.
\end{equation}
Consequently, by combining (\ref{Ed0}) with (\ref{es-1}), we also get following results:%
\begin{corollary}
	The following relationship holds true:
	\begin{equation*}
		\overset{x}{\underset{1}{\boldsymbol{S}}}\frac{1}{z^{2}+1}\underset{w}{%
			\Delta }z=L\left( 1,\chi \right) -\sum_{n=0}^{\infty }\frac{(-1)^{2n+1}}{%
			w^{2n+1}}\left( 1-\frac{1}{2^{1+2n}}\right) \frac{(2\pi )^{2n}}{2(2n)!}%
		B_{2n+2}\left( \frac{x}{w}\right).
	\end{equation*}
\end{corollary}
Consequently, by combining (\ref{Ed1}) with (\ref{es-1}), we also get following results:%
\begin{corollary}
	The following relationship holds true:
	\begin{equation*}
		\overset{x}{\underset{1}{\boldsymbol{S}}}\frac{1}{z^{2}+1}\underset{w}{%
			\Delta }z=\frac{1}{2}\zeta _{E}\left( 1,\frac{1}{2}\right)
		-\sum_{n=0}^{\infty }\frac{(-1)^{2n+1}}{w^{2n+1}}\left( 1-\frac{1}{2^{1+2n}}%
		\right) \frac{(2\pi )^{2n}}{2(2n)!}B_{2n+2}\left( \frac{x}{w}\right) .
	\end{equation*}
\end{corollary}
\subsection{Formulas involving special functions by applications Representation Theorem}

Here, we give some applications of the Representation Theorem and the homogeneous form of the N\"{o}rlund sum with (inverse) Laplace transforms.

By applying the Representation Theorem, given in (\ref{Rt}),\ to\ the
functions $f(z)$ given Table \ref{t3}, with the aid of the Laplace transform of the function $f(z)$, we obtain many formulas for $\overset{x}{\underset{a}{\boldsymbol{S}}}\mathcal{L}\left\{ f(z)\right\} 
\underset{w}{\Delta }z$.

\begin{table}[h!]
	
	\caption{Some values of the $\overset{x}{\underset{a}{\boldsymbol{S}}}\mathcal{L}\left\{ f(z)\right\} 
		\underset{w}{\Delta }z$}
	\centering
	{\small 
		\renewcommand{\arraystretch}{1.4} 
		\begin{tabular}{|c|l|}
			\hline
			\textbf{Function $f(z)$} & \textbf{$\overset{x}{\underset{a}{\boldsymbol{S}}}\mathcal{L}\left\{ f(z)\right\} 
				\underset{w}{\Delta }z$} \\  [5mm]
			\hline
			$\displaystyle f(z) = \frac{w}{z^{2}+w^{2}}$ & $\displaystyle 	\overset{x}{\underset{a}{\boldsymbol{S}}}\!\left(\frac{%
				w}{z^{2}+w^{2}}\right)\Delta z = \int_{0}^{\infty}\!\left(\frac{e^{-at}}{t}-%
			\frac{we^{-tx}}{1-e^{-wt}}\right)\sin(wt)\,dt$ \\  [5mm]
			\hline
			$\displaystyle f(z) = \frac{z}{z^{2}+w^{2}}$ & $\displaystyle 	\overset{x}{\underset{a}{\boldsymbol{S}}}\!\left(\frac{%
				z}{z^{2}+w^{2}}\right)\Delta z = \int_{0}^{\infty}\!\left(\frac{e^{-at}}{t}-%
			\frac{we^{-tx}}{1-e^{-wt}}\right)\cos(wt)\,dt$ \\ [5mm]
			\hline
			$\displaystyle f(z) = \frac{z}{z^{2}-b^{2}}$ & $\displaystyle 	\overset{x}{\underset{a}{\boldsymbol{S}}}\!\left(\frac{%
				z}{z^{2}-b^{2}}\right)\Delta z = \int_{0}^{\infty}\!\left(\frac{e^{-at}}{t}-%
			\frac{we^{-tx}}{1-e^{-wt}}\right)\cosh(bt)\,dt$ \\ [5mm] 
			\hline
			$\displaystyle f(z) = \frac{b}{z^{2}+b^{2}}$ & $\displaystyle 	\overset{x}{\underset{a}{\boldsymbol{S}}}\!\left(\frac{%
				b}{z^{2}+b^{2}}\right)\Delta z = \int_{0}^{\infty}\!\left(\frac{e^{-at}}{t}-%
			\frac{we^{-tx}}{1-e^{-wt}}\right)\sinh(bt)\,dt$ \\ [5mm]
			\hline
			$\displaystyle f(z) = \frac{1}{(z+a)^{v}}$ & $\displaystyle 	\overset{x}{\underset{a}{\boldsymbol{S}}}\!\left(\frac{1%
			}{(z+a)^{v}}\right)\Delta z = \int_{0}^{\infty}\!\left(\frac{e^{-at}}{t}-%
			\frac{we^{-tx}}{1-e^{-wt}}\right)\frac{t^{v-1}}{(v-1)!}e^{-at}\,dt$ \\[5mm]
			\hline
		\end{tabular}
	}
	\label{t3}
\end{table}
If we apply the Representation Theorem to any function with a Laplace transform, it is possible to obtain other formulas for $\overset{x}{\underset{a}{\boldsymbol{S}}}\mathcal{L}\left\{ f(z)\right\} 
\underset{w}{\Delta }z$, which are not included in Table \ref{t3}. For instance, substituting the following function $$f(z)=\frac{2w^{3}}{z^{2}\left( z^{2}+w^{2}\right) }$$ into (\ref%
{Rt}) and using fraction separation technique, we get
\begin{equation*}
	\overset{x}{\underset{a}{\boldsymbol{S}}}\frac{2w^{3}}{z^{2}\left(
		z^{2}+w^{2}\right) }\underset{w}{\Delta }z=\overset{x}{\underset{a}{%
			\boldsymbol{S}}}\frac{w}{z^{2}+w^{2}}\underset{w}{\Delta }z-w\overset{x}{%
		\underset{a}{\boldsymbol{S}}}\frac{z^{2}-w^{2}}{\left( z^{2}+w^{2}\right)
		^{2}}\underset{w}{\Delta }z.
\end{equation*}%
Using  Table \ref{t3} and the Laplace transform, we get%
\begin{eqnarray*}
	&&\overset{x}{\underset{a}{\boldsymbol{S}}}\frac{2w^{3}}{z^{2}\left(
		z^{2}+w^{2}\right) }\underset{w}{\Delta }z \\
	&=&\int\limits_{0}^{\infty }\left( \frac{e^{-at}}{t}-\frac{we^{-tx}}{%
		1-e^{-wt}}\right) \sin \left( wt\right) dt-w\int\limits_{0}^{\infty }\left( 
	\frac{e^{-at}}{t}-\frac{we^{-tx}}{1-e^{-wt}}\right) t\cos \left( wt\right)
	dt.
\end{eqnarray*}
We modify the homogenous N\"{o}rlund sum as follows: 
\begin{equation*}
	\overset{x}{\underset{x}{\boldsymbol{S}}}z^{v}\underset{w}{\Delta }z=\overset%
	{x}{\underset{a}{\boldsymbol{S}}}z^{v}\underset{w}{\Delta }%
	z-\int\limits_{a}^{x}z^{v}dz.
\end{equation*}%
Putting the function 
\begin{equation*}
	\varphi (z)=z^{v},
\end{equation*}
in the above equation and using the formula in (\ref{3.79}), after some
calculations in above the equation, we get the following theorem:

\begin{theorem}
	Let $v\geq 0$. Then we have%
	\begin{equation*}
		\overset{x}{\underset{x}{\boldsymbol{S}}}z^{v}\underset{w}{\Delta }z=\frac{%
			-a^{v+1}+w^{v+1}B_{v+1}(x;w)}{v+1}-\frac{x^{v+1}-a^{v+1}}{v+1}.
	\end{equation*}
\end{theorem}

\section{Conclusion}

In this paper, we derived new formulas using the N\"{o}rlund sum and the
Representation Theorem. Furthermore, we obtained novel formulas involving
the sum of (inverse) Laplace transforms, the alternating Hurwitz zeta
function, and the reciprocal Catalan sum formula with other operators. We
also derived new results by establishing the relationship between the
reciprocal Catalan sum formula and the derivative operator. Finally, we gave
examples of the relationship between these operators and the Dirichlet $L$-
function. By using the values of the Dirichlet character with conductor 4, we computed the values of the function  $L\left( 1,\chi \right)$ and the Dirichlet beta function.  By using these of the Dirichlet character with conductor 4, we also gave many useful Tables in order to compute special values of the N\"{o}rlund sum. We also introduced certain classes of infinite series, including
the Dirichlet $L$-function with Dirichlet character, the Hurwitz zeta
function, the alternating Hurwitz zeta function, and Bernoulli polynomials. 

Our future plane, by the aid of the Laplace transforms, other integral forms and mathematical tools, is to give find generalization of the N\"{o}rlund sum involving $p$-adic Volkenborn integral and indefinite sum on the set of $p$-adic integers and their cosets. We will also investigate relations among the N\"{o}rlund sum, the twisted $p$-adic $(h,q)$-$L$-functions, given in \cite{padic},  the Teichm\"{u}ller character, the Dirichlet characters, and other characters on multiplicative $p$-torsion subgroup, cyclic groups, etc.


\end{document}